\documentclass[11pt]{amsart}
\usepackage{amsmath}
\usepackage{amssymb}
\usepackage[all]{xy}
\usepackage{verbatim}

\def\IZ{{\Bbb Z}}

\def\PSL{\rm{PSL}}

\def\sideremark#1{\ifvmode\leavevmode\fi\vadjust{\vbox to0pt{\vss
  \hbox to 0pt{\hskip\hsize\hskip1em
  \vbox{\hsize2.5cm\tiny\raggedright\pretolerance10000
  \noindent #1\hfill}\hss}\vbox to8pt{\vfil}\vss}}}

\newtheorem{theorem}{Theorem}[section]
\newtheorem{lemma}[theorem]{Lemma}
\newtheorem{corollary}[theorem]{Corollary}

\newtheorem{prop}[theorem]{Proposition}
\theoremstyle{definition}

\renewcommand{\H}{\mathbb{H}}

\renewcommand{\phi}{\varphi}

\newcommand{\rs}{\widehat{\mathbb{C}}}

\begin{document}

\title[Moduli spaces of hyperbolic 3-manifolds]{Moduli spaces
of hyperbolic 3-manifolds and dynamics on character varieties}
\author{Richard D. Canary and Peter A. Storm}
\begin{abstract}
The space $AH(M)$ of marked hyperbolic 3-manifold homotopy equivalent
to a compact 3-manifold with boundary $M$ sits inside the 
${\rm PSL}_2({\bf C})$-character variety
$X(M)$ of $\pi_1(M)$. We study the dynamics of the action of
${\rm Out}(\pi_1(M))$ on both $AH(M)$ and $X(M)$. The nature of the dynamics
reflects the topology of $M$.

The quotient 
$AI(M)=AH(M)/{\rm Out}(\pi_1(M))$ may naturally be thought of as the moduli
space of unmarked hyperbolic 3-manifolds homotopy equivalent to $M$
and its topology reflects the dynamics of the action.
\end{abstract}

\thanks{The first author was partially supported by National Science Foundation
grants DMS-0504791 
and DMS-0554239. The second author was partially supported by National
Science Foundation grant DMS-0904355 and the
Roberta and Stanley Bogen Visiting Professorship at Hebrew
University..}
\date{\today}
\maketitle

\section{Introduction}\label{sec:introduction}
\setcounter{section}{1} 

For a compact, orientable, hyperbolizable $3$-manifold $M$ with boundary, the deformation space $AH(M)$ of marked hyperbolic 3-manifolds homotopy equivalent to $M$ is a familiar object of study. 
This deformation space sits naturally inside the ${\rm PSL}_2({\bf C})$-character variety $X(M)$ and
the outer automorphism group ${\rm Out}(\pi_1(M))$ acts by homeomorphisms on both $AH(M)$ and $X(M)$.  The action of ${\rm Out}(\pi_1(M))$
on $AH(M)$ and $X(M)$ has largely been studied in the case when $M$ is 
an interval bundle over a closed surface
(see, for example, \cite{Bow,Gold,SS,CS}) or a handlebody
(see, for example, \cite{minsky-primitive,Tan}). In this paper, we initiate a study of
this action for general hyperbolizable 3-manifolds.

We also study the topological quotient
$$AI(M) = AH(M) / {\rm Out}(\pi_1(M))$$
which we may think of as the moduli space of unmarked hyperbolic
3-manifolds homotopy equivalent to $M$.
The space $AH(M)$ is a rather pathological topological object itself, often
failing to even be locally connected (see Bromberg \cite{bromberg-PT} and
Magid \cite{magid}). However, since $AH(M)$ is a closed subset of an open
submanifold of the character variety, it does retain many nice topological
properties.  We will see that the topology of $AI(M)$ can be significantly more
pathological.

The first hint that the dynamics of ${\rm Out}(\pi_1(M))$ on $AH(M)$ are
complicated, was Thurston's  \cite{thurston2} proof that if $M$ is 
homeomorphic to \hbox{$S\times I$}, then there are infinite order elements of ${\rm Out}(\pi_1(M))$
which have fixed points in $AH(M)$. (These elements are pseudo-Anosov mapping
classes.)
One may further show that $AI(S\times I)$ is not even
$T_1$, see \cite{CS} for a closely related result.
Recall that a topological space is $T_1$ if all its points are closed.
On the other hand, we show that in all other cases $AI(M)$ is $T_1$.

\begin{theorem}
\label{T1thm}
Let $M$ be a compact hyperbolizable $3$-manifold with non-abelian fundamental group.  Then the moduli space $AI(M)$ is $T_1$
if and only if $M$ is not an untwisted interval bundle.
\end{theorem}

We next show that ${\rm Out}(\pi_1(M))$ does not act properly discontinuously
on $AH(M)$ if $M$ contains a primitive essential annulus.
A properly embedded annulus in $M$ is a primitive essential annulus
if it cannot be properly isotoped into the boundary of $M$ and its core
curve generates a maximal abelian subgroup of $\pi_1(M)$.
In particular, if $M$ has compressible boundary
and no toroidal boundary components, then $M$ contains a primitive essential annulus
(see Corollary \ref{cor:compressible_not_T2}).

\begin{theorem}
\label{prop:!T2}
Let $M$ be a compact hyperbolizable $3$-manifold with non-abelian
fundamental group.
If $M$ contains a primitive essential annulus then ${\rm Out}(\pi_1(M))$
does not act properly discontinuously on $AH(M)$. Moreover,
if $M$ contains a primitive essential annulus, then
$AI(M)$ is not Hausdorff.
\end{theorem}

On the other hand, if $M$ is acylindrical, i.e. has incompressible boundary and
contains no essential annuli, then ${\rm Out}(\pi_1(M))$ is finite
(see Johannson \cite[Proposition 27.1]{johannson}),  so 
${\rm Out}(\pi_1(M))$  acts properly discontinuously on $AH(M)$ and $X(M)$.
It is easy to see that ${\rm Out}(\pi_1(M))$ fails to act properly discontinuously
on $X(M)$ if $M$ is not acylindrical, since it will contain infinite order
elements with fixed points in $X(M)$.

If $M$ is a compact hyperbolizable 3-manifold which
is not acylindrical,  but does not contain
any primitive essential annuli, then ${\rm Out}(\pi_1(M))$ is infinite. However, if, in addition,
$M$ has no toroidal boundary components, we show that
${\rm Out}(\pi_1(M))$ acts properly discontinuously on an open neighborhood of
$AH(M)$ in $X(M)$. In particular, we see that $AI(M)$ is Hausdorff
in this case.

\begin{theorem}\label{thm:intro3}
If  $M$ is a compact hyperbolizable  3-manifold with no primitive
essential annuli
whose boundary has no toroidal boundary components, then there exists an open
${\rm Out}(\pi_1(M))$-invariant
neighborhood $W(M)$ of $AH(M)$ in $X(M)$ such  that ${\rm Out}(\pi_1(M))$ acts properly discontinuously on $W(M)$.
In particular, $AI(M)$ is Hausdorff.
\end{theorem}

If $M$ is a compact hyperbolizable  3-manifold with no primitive
essential annuli
whose boundary has no toroidal boundary components, then ${\rm Out}(\pi_1(M))$
is virtually abelian (see the discussion in sections \ref{sec:mod}
and \ref{propdisc}). However, we note that the conclusion of
Theorem \ref{thm:intro3} relies crucially on the topology of $M$,
not just the group theory of ${\rm Out}(\pi_1(M))$. In particular, if $M$
is a compact hyperbolizable 3-manifold $M$ with incompressible
boundary, such that every component of its characteristic submanifold is a solid torus,
then ${\rm Out}(\pi_1(M))$ is always virtually abelian, but $M$ may contain
primitive essential annuli, in which case ${\rm Out}(\pi_1(M))$ does
not act properly discontinuously on $AH(M)$.

\bigskip

One may combine Theorems \ref{prop:!T2} and \ref{thm:intro3} to
completely characterize when ${\rm Out}(\pi_1(M))$ acts properly discontinuously
on $AH(M)$ in the case that $M$ has no
toroidal boundary components.

\begin{corollary}\label{thm:intro2}
Let $M$ be a compact hyperbolizable $3$-manifold with no toroidal boundary components and non-abelian fundamental group.
The group ${\rm Out}(\pi_1(M))$ acts properly discontinuously on $AH(M)$ if and only if
$M$ contains no primitive essential annuli. Moreover,
$AI(M)$ is Hausdorff if and only if $M$ 
contains no primitive essential annuli.
\end{corollary}

It is a consequence of the classical deformation theory of Kleinian groups
(see Bers \cite{bers-survey} or Canary-McCullough \cite[Chapter 7]{CM} for a survey of this theory)
that ${\rm Out}(\pi_1(M))$ acts properly discontinuously on the interior
${\rm int}(AH(M))$ of $AH(M)$. 
If $H_n$ is the handlebody of genus $n\ge 2$, Minsky \cite{minsky-primitive} exhibited an
explicit ${\rm Out}(\pi_1(H_n))$-invariant open subset $PS(H_n)$ of $X(H_n)$ such
that ${\rm int}(AH(H_n))$ is a proper subset of $PS(H_n)$ and
${\rm Out}(\pi_1(H_n))$ acts properly discontinuously on $AH(H_n)$.

If $M$ is a compact
hyperbolizable 3-manifold with incompressible boundary and no toroidal
boundary components, which is not an interval bundle, then we find an open set $W(M)$
strictly bigger than ${\rm int}(AH(M))$  which ${\rm Out}(\pi_1(M))$ acts properly
discontinuosly on.
See Theorem \ref{Pdiscnbhd} and its proof for a more precise
description of $W(M)$. We further observe, see Lemma \ref{allofit}, that $W(M)\cap \partial AH(M)$ is
a dense  open subset of $\partial AH(M)$ in this setting.

\begin{theorem}
\label{openpd}
Let $M$ be a compact hyperbolizable $3$-manifold with nonempty incompressible boundary and no toroidal boundary components, which is not an
interval bundle. Then there exists an open ${\rm Out}(\pi_1(M))$-invariant subset 
$W(M)$ of  $X(M)$ such that ${\rm Out}(\pi_1(M))$  acts properly discontinuously on $W(M)$
and ${\rm int}(AH(M))$ is a proper subset of $W(M)$.
\end{theorem}

It is conjectured  that if $M$ is an untwisted interval bundle over a closed
surface $S$,
then ${\rm int}(AH(M))$ is the maximal open ${\rm Out}(\pi_1(M))$-invariant subset of $X(M)$ on which ${\rm Out}(\pi_1(M))$ acts properly discontinuously.  One may show that no open
domain of discontinuity can intersect $\partial AH(S\times I)$ (see \cite{michelle}).
Further evidence for this
conjecture is provided by results of Bowditch \cite{Bow},
Goldman \cite{GoldMer}, Souto-Storm
\cite{SS}, Tan-Wong-Zhang \cite{Tan} and Cantat \cite{cantat}.

Michelle Lee \cite{michelle} has recently shown that if $M$ is an twisted
interval bundle over a closed surface, then there exists an open ${\rm Out}(\pi_1(M))$-invariant
subset $W$ of $X(M)$ such that ${\rm Out}(\pi_1(M))$ acts properly discontinuously on
$W$ and ${\rm int}(AH(M))$ is a proper subset of $W$. Moreover, $W$ contains
points in $\partial AH(M)$. As a corollary, she proves that if $M$ has incompressible
boundary and no toroidal boundary components, then there is open
${\rm Out}(\pi_1(M))$-invariant
subset $W$ of $X(M)$ such that ${\rm Out}(\pi_1(M))$ acts properly discontinuously on
$W$, ${\rm int}(AH(M))$ is a proper subset of $W$, and 
$W\cap\partial AH(M)\ne \emptyset$ if and only if $M$ is not an untwisted interval bundle.

\bigskip\noindent
{\bf Outline of paper:} In section 2, we recall background material from topology
and hyperbolic geometry which will be used in the paper.

In section 3, we prove Theorem \ref{T1thm}. The proof
that $AI(S\times I)$ is not $T_1$ follows the arguments in
\cite[Proposition 3.1]{CS} closely.
We now sketch the proof that $AI(M)$ is $T_1$  otherwise.
In this case, let $N\in AI(M)$ and let $R$ be
a compact core for $N$. We show that $N$ is a closed point, by showing
that any convergent sequence $\{\rho_n\}$ 
in the pre-image of $N$ is eventually constant. For all $n$, there exists a homotopy
equivalence $h_n:M\to N$ such that $(h_n)_*=\rho_n$. If $G$ is a graph in $M$
carrying $\pi_1(M)$, then, since $\{\rho_n\}$ is convergent, we can assume that
the length of $h_n(G)$ is at most $K$, for all $n$ and some $K$. But,
we observe that $h_n(G)$ cannot lie entirely in the complement of $R$, if $R$ is
not a compression body.
In this case, each $h_n(G)$ lies in the compact neighborhood of radius $K$ of $R$, so
there are only finitely many possible homotopy classes of maps of $G$.
Thus, there are only
finitely many possibilities for $\rho_n$, so $\{\rho_n\}$ is eventually constant.
The proof in the case that $R$ is a compression body is somewhat more complicated
and uses the Covering Theorem.

In section 4, we prove Theorem \ref{prop:!T2}. Let $A$ be a primitive essential
annulus in $M$. If $\alpha$ is a core curve of $A$, then the complement $\hat M$ of
a regular neighborhood of $\alpha$ in $M$ is hyperbolizable. We consider a
geometrically finite  hyperbolic manifold $\hat N$ homeomorphic to the interior of $\hat M$ 
and use the Hyperbolic Dehn Filling Theorem to produce a convergent
sequence $\{\rho_n\}$ in $AH(M)$ and a sequence $\{\phi_n\}$ of distinct elements of
${\rm Out}(\pi_1(M))$ such that
$\{\rho_n\circ\phi_n\}$ also converges.
Therefore, ${\rm Out}(\pi_1(M))$ does not act properly discontinuously on $AH(M)$.
Moreover, we show that $\{\rho_n\}$ projects to a sequence in $AI(M)$ with two distinct limits,
so $AI(M)$ is not Hausdorff.

In section 5 we recall basic facts about the characteristic submanifold and
the mapping class group of compact hyperbolizable 3-manifolds with incompressible
boundary and
no toroidal boundary components. We identify a finite index subgroup $J(M)$ of
${\rm Out}(\pi_1(M))$ and a projection of $J(M)$ onto the direct product of mapping class groups of
the base surfaces whose kernel $K(M)$ is the free abelian subgroup generated by Dehn twists in frontier annuli
of the characteristic submanifold.

In section 6, we organize the frontier annuli of the characteristic submanifold
into characteristic collections of annuli and describe free subgroups of $\pi_1(M)$
which register the action of the subgroup of ${\rm Out}(\pi_1(M))$ generated by
Dehn twists in the annuli in such a collection.

In section 7, we show that compact hyperbolizable 3-manifolds with compressible
boundary and no toroidal boundary components contain primitive essential annuli.

In section 8, we introduce a subset $AH_n(M)$ of $AH(M)$ which
contains all purely hyperbolic representations. We see that ${\rm int}(AH(M))$ is a 
proper subset of $AH_n(M)$ and that
$AH_n(M)=AH(M)$ if $M$ does not contain any primitive essential annuli.

In section 9, we prove that if $M$ has incompressible boundary and no
toroidal boundary components, but is not an interval bundle, there is
an open neighborhood $W(M)$ of $AH_n(M)$ in $X(M)$ such that 
${\rm Out}(\pi_1(M))$ preserves and acts properly discontinuously on $W(M)$.
Theorems \ref{thm:intro3} and \ref{openpd} are  immediate corollaries.
We finish the outline by sketching the proof in a special case. 

Let $X$ be an acylindrical, compact hyperbolizable 3-manifold
and let $A$ be an incompressible annulus in its boundary.
Let $V$ be a solid torus and let $B_1,\ldots,B_n$ be a collection
of disjoint parallel annuli in $\partial V$ whose core curves 
are homotopic to the $n^{th}$ power of the core curve of $V$
where $|n|\ge 2$.
Let $M_1,\ldots, M_n$ be copies of $X$ and let $A_1,\ldots, A_n$
be copies of $A$ in $M_i$.
We form $M$ by attaching each $M_i$ to $V$ by identifying
$A_i$ and $B_i$. Then $M$ contains no primitive essential annuli,
is hyperbolizable, and ${\rm Out}(\pi_1(M))$
has a finite index subgroup $J(M)$ generated by Dehn twists about
$A_1,\ldots, A_n$. In particular,
$J(M)\cong {\bf Z}^{n-1}.$

In this case, $\{A_1,\ldots,A_n\}$ is the only characteristic collection of annuli.
We say that  a group $H$ registers $J(M)$ if it is freely generated by
the core curve of $V$ and, for each $i$, a curve contained in $V\cup M_i$ 
which is not homotopic into $V$. So $H\cong F_{n+1}$.
There is a natural map $r_H:X(M)\to X(H)$ where $X(H)$ is the
${\rm PSL}_2({\bf C})$-character variety of the group $H$.
Notice that $J(M)$ preserves $H$ and injects into ${\rm Out}(H)$.
Let 
$$\mathcal{S}_{n+1}={\rm int}(AH(H))\subset X(H)$$
denote the space of  Schottky representations (i.e. representations
which are purely hyperbolic and geometrically finite.)
Since ${\rm Out}(H)$ acts properly discontinuously on $\mathcal{S}_{n+1}$,
we see that $J(M)$ acts properly discontinuously on
$$W_H=r_H^{-1}(\mathcal{S}_{n+1})$$
Let $W(M)=\bigcup W_H$ where the union is taken over all
subgroups which register $J(M)$. Notice that $W(M)$ is open
and $J(M)$ acts properly discontinuously on $W(M)$. One may use a ping
pong argument to show that $AH_n(M)\subset W(M)$, see Lemma \ref{CjSchottky}. Johannson's
Classification Theorem is used to show that $W(M)$ is invariant under
${\rm Out}(\pi_1(M))$, see Lemma \ref{controlK}. (Actually,
we define a somewhat larger set, in general, by
using the space of primitive-stable representations in place of Schottky space.)

\bigskip\noindent
{\bf Acknowledgements:} Both authors would like to thank Yair Minsky
for helpful and interesting conversations.
The second author thanks the organizers of the workshop ``Dynamics of $Aut(F_n)$-actions on representation varieties,'' held in Midreshet Sde Boker, Israel in January, 2009. We thank
Michelle Lee, Darryl McCullough and the referee for helpful comments on a preliminary version of this paper.

\section{Preliminaries}\label{sec:prelim}

As a convention, the letter $M$ will denote a compact connected oriented hyperbolizable $3$-manifold with boundary.  We recall that $M$ is said to be hyperbolizable if 
the interior of $M$ admits a complete hyperbolic metric. 
We will use $N$ to denote a hyperbolic $3$-manifold.  
All hyperbolic $3$-manifolds are assumed to be oriented, complete, and connected. 
  
\subsection{The deformation spaces}

Recall that $\PSL_2 ({\bf C})$ is the group of orientation-preserving isometries of
${\bf H}^3$.  Given a $3$-manifold $M$, a discrete, faithful representation
$\rho : \pi_1(M) \to \PSL_2 ({\bf C})$ determines a hyperbolic $3$-manifold
$N_\rho = {\bf H}^3 / \rho(\pi_1(M))$ and a homotopy equivalence
$m_\rho : M \to N_\rho$, called the marking of $N_\rho$. 

We let $D(M)$ denote the set of discrete, faithful representations
of $\pi_1(M)$ into $\PSL_2({\bf C})$. The group $\PSL_2 ({\bf C})$ acts
by conjugation on $D(M)$ and we let
$$AH(M)=D(M)/\PSL_2({\bf C}).$$ 
Elements of $AH(M)$ are hyperbolic 3-manifolds
homotopy equivalent to $M$ equipped with (homotopy classes of) markings.

The space $AH(M)$ is a closed subset of
the character variety
$$X(M)={\rm Hom}_T(\pi_1(M), \PSL_2 ({\bf C}))//\PSL_2({\bf C}),$$
which is the Mumford quotient of the space
${\rm Hom}_T(\pi_1(M), \PSL_2 ({\bf C}))$  of representations
$\rho:\pi_1(M)\to \PSL_2 ({\bf C})$ such that $\rho(g)$ is parabolic if $g\ne id$
lies in a rank two free abelian subgroup of $\pi_1(M)$.
If $M$ has no toroidal boundary components, then
${\rm Hom}_T(\pi_1(M), \PSL_2 ({\bf C}))$ is simply ${\rm Hom}(\pi_1(M), \PSL_2 ({\bf C}))$.
Moreover, each point in $AH(M)$ is a smooth point of $X(M)$
(see Kapovich \cite[Sections 4.3 and 8.8]{Kap} and Heusener-Porti \cite{HP} for more
details on this construction). 

The group $Aut(\pi_1(M))$ acts naturally on ${\rm Hom}_T(\pi_1(M), \PSL_2({\bf C}))$ via
\[ ( \phi \cdot \rho) (\gamma) := \rho (\phi^{-1} (\gamma)).\]
This descends to an action of ${\rm Out}(\pi_1(M))$ on $AH(M)$ and $X(M)$.  This action is not free, and it often has complex dynamics.  Nonetheless, we can define the topological quotient space 
$$AI(M) = AH(M) /{\rm Out}(\pi_1(M)).$$ 
Elements of $AI(M)$ are naturally oriented hyperbolic 3-manifolds homotopy equivalent to
$M$ without a specified marking. 

\subsection{Topological background}

A compact 3-manifold $M$ is said to have {\em incompressible boundary}
if whenever $S$ is a component of $\partial M$, the inclusion map induces an injection of
$\pi_1(S)$ into $\pi_1(M)$. In our setting, this is equivalent to $\pi_1(M)$
being freely indecomposable. A properly embedded annulus $A$  in $M$ is
said to be {\em essential} if the inclusion map induces an injection of $\pi_1(A)$
into $\pi_1(M)$ and $A$ cannot be properly homotoped into $\partial M$ (i.e.
there does not exist a homotopy of pairs of the inclusion $(A,\partial A)\to (M,\partial M)$
to a map with image in $\partial M$). An essential annulus $A$ is said to be {\em primitive} if the image of $\pi_1(A)$ in $\pi_1(M)$ is a maximal abelian subgroup.

If $M$ does not have incompressible boundary, it is said to have
{\em compressible boundary.} The fundamental examples of 3-manifolds
with compressible boundary are compression bodies.
A {\em compression body} is either a handlebody or is formed 
by attaching  $1$-handles to disjoint disks on the 
boundary surface $R \times \{1\}$ of a 3-manifold $R\times [0,1]$ where
$R$ is a closed, but not necessarily connected, surface
(see, for example, Bonahon \cite{Bo}).
The resulting $3$-manifold $C$ (assumed to be connected) will have a single boundary component
$\partial_+ C$ intersecting $R \times \{1\}$, called the positive (or external) boundary of $C$.
If $C$ is not an untwisted interval bundle over a closed surface, then 
$\partial_+C$ is the unique compressible boundary component of $C$.
Notice that the induced homomorphism
$\pi_1(\partial_+ C) \to \pi_1(C)$ is surjective.  In fact, a compact irreducible 3-manifold 
$M$ is  a compression body if and only if there exists a component $S$ of $\partial M$
such that $\pi_1(S)\to\pi_1(M)$ is surjective.

Every compact hyperbolizable 3-manifold can be
constructed from compression bodies and manifolds with incompressible boundary.
Bonahon \cite{Bo} and McCullough-Miller \cite{mccullough-miller} showed that
there exists a neighborhood $C_M$ of
$\partial M$, called the {\em characteristic compression body},
such that each component of $C_M$ is a compression body and 
each component of  $\partial C_M-\partial M$ is incompressible in $M$.

Dehn filling will play a key role in the proof of Theorem \ref{prop:!T2}.
Let $F$ be a toroidal boundary component of  compact 3-manifold $M$
and let $(m,l)$ be  a choice of  meridian and longitude for $F$.
Given a pair $(p,q)$ of relatively
prime integers, we may form a new manifold $M(p,q)$ by attaching a solid torus $V$ 
to $M$ by an orientation-reversing homeomorphism $g\colon\, \partial V \to F$
so that, if $c$ is the meridian of $V$, then $g(c)$ is a $(p,q)$
curve on $F$ with respect to the chosen meridian-longitude system.
We say that $M(p,q)$ is obtained from $M$ by {\em
$(p,q)$-Dehn filling along $F$}.

\subsection{Hyperbolic background}

If $N=\H^3/\Gamma$ is a hyperbolic 3-manifold, then $\Gamma\subset\PSL_2({\bf C})$
acts on $\rs$ as a group of conformal automorphisms. The {\em domain of discontinuity}
$\Omega(\Gamma)$ is the largest open $\Gamma$-invariant subset of $\rs$ on
which $\Gamma$ acts properly discontinuously. Note that $\Omega(\Gamma)$
may be empty. Its complement $\Lambda(\Gamma)=\rs-\Omega(\Gamma)$ is
called the {\em limit set.} The quotient $\partial_cN=\Omega(\Gamma)/\Gamma$
is naturally a Riemann surface called the {\em conformal boundary}.

Thurston's Hyperbolization theorem, see Morgan \cite[Theorem $B'$]{Morgan},
guarantees that if $M$ is compact and hyperbolizable, then there exists
a hyperbolic 3-manifold $N$ and a homeomorphism
$$\psi:M-\partial_TM\to N\cup\partial_cN$$
where $\partial_TM$ denotes
the collection of toroidal boundary components of $M$.

The convex core $C(N)$ of $N$ is the smallest convex submanifold 
whose inclusion into $N$ is a homotopy equivalence.
More concretely, it is obtained as the quotient, by $\Gamma$, of the convex
hull, in $\H^3$, of the limit set $\Lambda(\Gamma)$.
There is a well-defined retraction $r:N\to C(N)$ obtained by taking $x$ to
the (unique) point in $C(N)$ closest to $x$. The nearest point retraction
$r$ is a homotopy equivalence and is ${1\over\cosh s}$-Lipschitz on the complement
of the neighborhood of radius $s$ of $C(N)$. 

There exists a universal constant $\mu$, called the Margulis constant, such that if $\epsilon<\mu$,
then each component of the $\epsilon$-thin part
$$N_{thin(\epsilon)}=\{x\in N\ |\ {\rm inj}_N(x)<\epsilon\}$$
(where ${\rm inj}_N(x)$ denotes the injectivity radius of $N$ at $x$)
is either a metric regular neighborhood of a geodesic or is homeomorphic
to $T\times (0,\infty)$ where $T$ is either a torus or an open annulus
(see Benedetti-Petronio \cite{BP} for example).
The $\epsilon$-thick part of $N$ is defined simply to be the complement
of the $\epsilon$-thin part
$$N_{thick(\epsilon)}=N-N_{thin(\epsilon)}.$$ It is also useful to
consider the manifold $N^0_\epsilon$ obtained from $N$ by
removing the non-compact components of $N_{thin(\epsilon)}$.

If $N$ is a hyperbolic 3-manifold with finitely generated fundamental group,
then it admits a compact core, i.e. a compact submanifold whose
inclusion into $M$ is a homotopy equivalence (see Scott \cite{scott}).
More generally, if $\epsilon<\mu$, then there exists a {\em relative
compact core} $R$ for $N^0_\epsilon$, i.e. a compact core which intersects
each component of $\partial N^0_\epsilon$ in a compact core for
that component
(see Kulkarni-Shalen \cite{KS} or McCullough \cite{McC}).
Let $P=\partial R-\partial N^0_\epsilon$ and let $P^0$ denote
the interior of $P$.
The Tameness Theorem of Agol \cite{agol} and Calegari-Gabai \cite{calegari-gabai}
assures us that we may choose $R$ so that $N^0_\epsilon-R$ is homeomorphic
to $(\partial R-P^0)\times (0,\infty)$. In particular, the ends of $N^0_\epsilon$ 
are in one-to-one correspondence with the components of $\partial R-P^0$.
(We will blur this distinction and simply regard an end as a component
of $N^0_\epsilon-R$ once we have chosen $\epsilon$ and a relative compact
core $R$ for $N^0_\epsilon$.)
We say that an end $U$ of $N^0_\epsilon$ is {\em geometrically finite} if the intersection
of $C(N)$ with $U$ is bounded (i.e. admits a compact closure). $N$ is said
to be geometrically finite if all the ends of $N^0_\epsilon$ are geometrically finite.

Thurston \cite{thurston-notes} showed that if $M$ is a compact hyperbolizable
3-manifold whose boundary is a torus $F$, then all but finitely many
Dehn fillings of $M$ are hyperbolizable. Moreover, as the Dehn surgery
coefficients approach $\infty$, the resulting hyperbolic manifolds ``converge''
to the hyperbolic 3-manifold homeomorphic to ${\rm int}(M)$. If $M$ has other boundary
components, then there is a version of this theorem where one
begins with a geometrically finite hyperbolic 3-manifold homeomorphic to
${\rm int}(M)$  and one is allowed
to perform the Dehn filling while fixing
the conformal structure on the non-toroidal boundary components of $M$.
The proof uses the cone-manifold deformation theory developed by
Hodgson-Kerckhoff \cite{HK} in the finite volume case and extended
to the infinite volume case by Bromberg \cite{bromberg-deform}
and Brock-Bromberg \cite{brock-bromberg}. 
(The first statement of a Hyperbolic Dehn Filling Theorem in the
infinite volume setting was given by Bonahon-Otal \cite{BO}, see also
Comar \cite{comar}.)
For a general statement of the Filling Theorem, and a discussion of its derivation from
the previously mentioned work, see Bromberg \cite{bromberg-PT} or
Magid \cite{magid}.

\medskip\noindent
{\bf Hyperbolic Dehn Filling Theorem:} {\em
Let $M$ be a compact, hyperbolizable $3$-manifold and let $F$ be a toroidal
boundary component of $M$.
Let $N=\H^3/\Gamma$ be a hyperbolic $3$-manifold admitting an
orientation-preserving homeomorphism
$\psi:M-\partial_TM\to N\cup\partial_cN$.
Let $\{(p_n,q_n)\}$ be an infinite sequence of distinct 
pairs of relatively prime integers.

Then, for all sufficiently large $n$, there exists a (non-faithful)  representation
$\beta_n:\Gamma\to {\rm PSL}_ 2({\bf C})$ with discrete image such that
\begin{enumerate}
\item $\{ \beta_n\}$ converges to the identity representation of
$\Gamma$, and
\item if $i_n: M\to M(p_n,q_n)$
denotes the inclusion map, then for each $n$, there exists an
orientation-preserving homeomorphism 
$$\psi_n: M(p_n,q_n)-\partial_TM(p_n,q_n) \to N_{\beta_n}\cup\partial_cN_{\beta_n}$$
such that
$\beta_n\circ \psi_*$ is conjugate to $(\psi_n)_*\circ (i_n)_*$, and
the restriction of $\psi_n\circ i_n\circ \psi^{-1}$ to $\partial_cN$ is
conformal.
\end{enumerate}
}

\section{Points are usually closed}\label{sec:T1}

If $S$ is a closed orientable surface, we showed in \cite{CS} that 
$\mathcal{AI}(S)=AH(S\times I)/{\rm Mod}_+(S)$ is not $T_1$ where ${\rm Mod}_+(S)$ is
the group of (isotopy classes of) orientation-preserving homeomorphisms of $S$. We recall that a topological space is $T_1$ if
all points are closed sets.
Since ${\rm Mod}_+(S)$ is identified with an
index two subgroup of ${\rm Out}(\pi_1(S))$, one also expects that
$AI(S\times I)=AH(S\times I)/{\rm Out}(\pi_1(S))$ is not $T_1$. 

In this section,
we show that if $M$ is an untwisted interval bundle, which also includes
the case that $M$ is a handlebody, then $AI(M)$ is not $T_1$, but that
$AI(M)$ is $T_1$ for all other compact, hyperbolizable 3-manifolds.

\medskip\noindent
{\bf Theorem \ref{T1thm}.} {\em
Let $M$ be a compact hyperbolizable $3$-manifold with non-abelian fundamental group.  Then the moduli space $AI(M)$ is $T_1$
if and only if $M$ is not an untwisted interval bundle.
}

\medskip

\begin{proof}
We first show that $AI(M)$ is $T_1$ if $M$ is not an untwisted interval bundle.
Let $p : AH(M) \to AI(M)$ be the quotient map and let $N$ be a hyperbolic manifold in $AI(M)$.  
We must show that $p^{-1}(N)$ is a closed subset of $AH(M)$. Since $AH(M)$ is
Hausdorff and second countable, it suffices to show that if
$\{ \rho_n \}$ is a convergent sequence in $p^{-1}(N)$, then $\lim \rho_n \in p^{-1}(N)$.

An element $\rho\in p^{-1}(N)$ is a representation such that $N_\rho$ is
isometric to $N$. Let $\{\rho_n\}$ be a convergent sequence of
representations in $p^{-1}(N)$.
Let  $G \subset M$ be a finite graph such that the inclusion map induces a surjection
of $\pi_1(G)$ onto $\pi_1(M)$. 
Each $\rho_n $ gives rise to a homotopy equivalence $h_n:M \to N$, and hence to a map $j_n=h_n|_G : G \to N$, both of which are  only well-defined up to homotopy.
Since $\{ \rho_n \}$ is convergent, there exists $K$ such that $j_n (G)$ has length at 
most $K$ for all $n$, after possibly altering $h_n$ by a homotopy.

Let $R$ be a compact core for $N$.
Assume first that $R$ is not a compression body.  In this case, if $S$ is any component
of $\partial R$, then  the inclusion map does not induce a surjection of $\pi_1(S)$
to $\pi_1(R)$  (see the discussion in section \ref{sec:prelim}).  Since $j_n (G)$ carries the fundamental group it cannot lie entirely outside of $R$.  It follows that $j_n (G)$ lies in the closed neighborhood $\mathcal{N}_K(R)$
of radius $K$ about $R$.   By compactness, there are only finitely many homotopy classes of maps of $G$ into $\mathcal{N}_K(R)$ with total length at most $K$. 
Hence, there are only finitely many different representations among the $\rho_n$, up to conjugacy.  The deformation space $AH(M)$ is Hausdorff, and the sequence
$\{ \rho_n \}$ converges, implying that $\{ \rho_n \}$ is eventually constant.  
Therefore $\lim \rho_n$ lies in the preimage of $N$,
implying that the fiber $p^{-1}(N)$ is closed and that $N$ is a closed point of $AI(M)$.

Next we assume that $R$ is a compression body. If
$R$ were an untwisted  interval bundle, then $M$ would also have to be a untwisted interval
bundle (by Theorems 5.2 and 10.6 in Hempel \cite{hempel}) which we have disallowed. 
So $R$ must have at least one incompressible boundary component and only
one compressible boundary component $\partial_+R$.
We are free to assume that $M$ is homeomorphic to $R$, 
since the definition of $AI(M)$ depends only on the homotopy
type of $M$. Let $D$ denote the union of $R$ and the component of $N-R$ bounded
by $\partial_+R$. Since the fundamental group of a component of $N-D$
never surjects onto $\pi_1(N)$, with respect to the map induced by inclusion, we
see as above that each $j_n(G)$ must intersect $D$, so is contained
in the neighborhood of radius $K$ of $D$.

Recall that there exists $\epsilon_K>0$ so that the distance from the
$\epsilon_K$-thin part of $N$ to the $\mu$-thick part of $N$ is greater
than $K$ (where $\mu$ is the Margulis constant). It follows that $j_n(G)$
must be contained in the $\epsilon_K$-thick part of $N$.

Let $F$ be an incompressible boundary component of $M$.
Then $h_n(F)$ is homotopic to an incompressible boundary component of $R$
(see, for example, the proof of Proposition 9.2.1 in \cite{CM}).
As there are finitely many possibilities, we may pass to a subsequence
so that $h_n(F)$ is homotopic to a fixed boundary component $F'$.
We may choose $G$ so that there is a proper subgraph $G_F \subset G$ 
such that the image of  $\pi_1(G_F)$ in $\pi_1(M)$  (under the inclusion map) is
conjugate to $\pi_1(F)$. 
Let \hbox{$p_F:N_F\to N$} be the covering map associated to $\pi_1(F')\subset\pi_1(N)$. 
Then $j_n |_{G_F}$ lifts to a map $k_n$ of $G_F$ into $N_F$.

Assume first that $F$ is a torus. Then $k_n(G_F)$ must lie in the portion $X$ of $N_F$
with injectivity radius between $\epsilon_K$ and $K/2$, which is compact.
It follows that $j_n(G)$
must lie in the closed neighborhood of radius $K$ of $p_F(X)$.
Since $p_F(X)$ is compact, we may conclude, as in the general case,
that $\{\rho_n\}$ is eventually constant and hence that $p^{-1}(N)$ is closed.

We now suppose that $F$ has genus at least $2$. We first establish that
there exists $L$ such that $k_n(G_F)$ must be contained in a neighborhood
of radius $L$ of the convex core $C(N_F)$. It is a consequence of the
thick-thin decomposition, that if $G'$ is a graph in $N_F$ which carries
the fundamental group then $G'$ must have length at least $\mu$.
We also recall that the nearest point retraction $r_F:N_F\to C(N_F)$ is a homotopy
equivalence which is ${1\over \cosh s}$-Lipschitz on the complement 
of the neighborhood of radius $s$ of $C(N)$.
Therefore, if $k_n(G_F)$ lies outside of $\mathcal{N}_s(C(N_F))$, then
$r_F(k_n(G_F))$  has length at most $K\over \cosh s$.
It follows that $k_n(G_F)$ must intersect the neighborhood of
radius $\cosh^{-1}({K\over\mu})$
of $C(N_F)$, so we may choose $L=K+\cosh^{-1}({K\over\mu})$.
 
If $N_F$ is geometrically
finite, then $X=C(N_F)\cap N_{thick(\epsilon_K)}$ is compact and
$j_n(G)$ must be contained in the neighborhood of radius $L+K$ of $p_F(X)$
which allows us to complete the proof as before.
 
If $N_F$ is not geometrically
finite, we will need to invoke the Covering Theorem to complete the proof.
Let $\tilde F$ denote
the lift of $F'$ to $N_F$. Then $\tilde F$ divides $N_F$ into two components,
one of which, say $A_-$, is mapped homeomorphically to the component of $N-R$
bounded by $F'$. Let $A_+=N_F-A_-$. We may choose a  
a relative compact core $R_F$ for $(N_F)^0_\epsilon$ (for some $\epsilon<\epsilon_K$)
so that $\tilde F$
is contained in the interior of $R_F$. Since $p_F$ is infinite-to-one on each
end of $(N_F)^0_\epsilon$ which is contained in $A_+$, the Covering Theorem
(see \cite{cover} or \cite{thurston-notes})
implies that all such ends are geometrically finite. Therefore,
$$Y=A_+\cap  C(N_F)\cap (N_F)_{thick(\epsilon_K)}$$
is compact.
If we let $Z=A_-\cup Y$, then we see that $k_n(G_F)$ is contained in  the closed neighborhood of radius $L$ about $Z$
(since $C(N_F)\cap N_{thick(\epsilon_K)}\subset Z$).
Therefore, $j_n(G)$ is contained in the closed
$(L+K)$-neighborhood of $ D\cap p_F(Z)=D\cap p_F(Y)$. 
Since $D\cap p_F(Y)$ is compact,
we conclude, exactly as in the previous cases, that $p^{-1}(N)$ is closed.
This case completes the proof that $AI(M)$ is $T_1$ if $M$ is not an untwisted
interval bundle.

\bigskip

We now deal with the case where $M=S\times I$ is an untwisted interval bundle over
a compact surface $S$.  (In the special case that $M$ is a handlebody of genus 2, we choose
$S$ to be the punctured torus.) In our previous paper \cite{CS}, we consider the space
$AH(S)$ of (conjugacy classes of) discrete faithful representations 
$\rho:\pi_1(S)\to \PSL_2({\bf C})$ such
that if $g\in\pi_1(S)$ is peripheral, then $\rho(g)$ is parabolic.
In Proposition 3.1, we use work of Thurston \cite{thurston2} and McMullen 
\cite{mcmullen} to exhibit a sequence $\{\rho_n\}$ in $AH(S)$ which
converges to $\rho\in AH(S)$ such that $\Lambda(\rho)=\rs$,
$\Lambda(\rho_1)\ne \rs$ and
for all $n$ there exists $\phi_n\in {\rm Mod(S)}$
such that $\rho_n=\rho_1\circ\phi_n$. Since $AH(S)\subset AH(S\times I)$
and ${\rm Mod}(S)$ is identified with a subgroup of ${\rm Out}(\pi_1(S))$,
we see that $\{\rho_n\}$ is a sequence in $p^{-1}(N_{\rho_1})$ which converges
to a point outside of $p^{-1}(N_{\rho_1})$. Therefore, $N_{\rho_1}$ is a point
in $AI(S\times I)$ which is not closed.
\end{proof}

\medskip\noindent
{\bf Remark:} One may further show,
as in the remark after Proposition 3.1 in \cite{CS},
that if $N\in AI(S\times I)$ is a degenerate hyperbolic
3-manifold with a lower bound on its injectivity radius, then $N$ is not a closed
point in $AI(S\times I)$. We recall that $N=\H^3/\Gamma$ is degenerate if
$\Omega(\Gamma)$ is connected and simply connected and $\Gamma$ is
finitely generated.

\section{Primitive essential annuli and the failure of proper discontinuity}

In this section, we show that 
if $M$ contains a primitive essential annulus, then ${\rm Out}(\pi_1(M))$ does
not act properly discontinuously on  $AH(M)$. 
We do so by using the Hyperbolic Dehn Filling Theorem to produce
a convergent sequence $\{\rho_n\}$ in $AH(M)$ and a sequence $\{\phi_n\}$  of
distinct element of ${\rm Out}(\pi_1(M))$ such that that $\{\rho_n\circ \phi_n\}$ is also convergent.
The construction is a generalization of a construction of Kerckhoff-Thurston
\cite{KT}. One may also think of the argument as a simple version of the
``wrapping'' construction (see Anderson-Canary \cite{ACpages})
which was also used to show that components of ${\rm int}(AH(M))$
self-bump whenever $M$ contains a primitive essential annulus
(see McMullen \cite{mcmullenCE} and Bromberg-Holt \cite{BH}).

\medskip\noindent
{\bf Theorem \ref{prop:!T2}.} {\em
Let $M$ be a compact hyperbolizable $3$-manifold with non-abelian
fundamental group.
If $M$ contains a primitive essential annulus then ${\rm Out}(\pi_1(M))$
does not act properly discontinuously on $AH(M)$. Moreover,
if $M$ contains a primitive essential annulus, then
$AI(M)$ is not Hausdorff.
}

\begin{proof}
Let $A$ be a primitive essential annulus in $M$ with core curve $\alpha$. 
Let $\hat{M} = M - \mathcal{N}(\alpha)$ where $\mathcal{N}(\alpha)$ is an open
regular neighborhood of $\alpha$. Lemma 10.2 in \cite{ACM} observes
that $\hat M$ is hyperbolizable. Since $\hat M$ is hyperbolizable, Thurston's
Hyperbolization Theorem implies that there
exists a  hyperbolic manifold $\hat N$ and a homeomorphism
$\psi:\hat{M}-\partial_T\hat{M}\to \hat{N}\cup\partial_c\hat{N}$.
The classical deformation theory of Kleinian groups (see Bers \cite{bers-survey} or \cite{CM}) implies
that we may choose any conformal structure on $\partial_c\hat N$.

Let $A_0$ and $A_1$ denote the components of $A\cap\hat M$. Let $M_i$
be the complement in $\hat M$ of a regular neighborhood of $A_i$.
Let $h_i: M \to \hat{M}$ be an embedding with image $M_i$ which
agrees with the identity map off of a (somewhat larger) regular neighborhood of $A$. 

Let $F$ be the toroidal boundary component of $\hat M$ which is
the boundary of $\mathcal{N}(\alpha)$ in $M$. Choose a meridian-longitude
system for $F$ so that the meridian for $F$ bounds a disk in $M$ and the longitude is
isotopic to $A_1\cap F$.
Lemma 10.3 in \cite{ACM} implies that if $i_n:\hat M\to \hat M(1,n)$ is
the inclusion map, then $i_n\circ h_i:M\to \hat M(1,n)$ is homotopic to a
homeomorphism for each $i=1,2$ and all $n\in\IZ$. Moreover,
we may similarly check that $i_n\circ h_1$ is homotopic to $i_n\circ h_0\circ D_A^n$ for all $n$,
where $D_A$ denotes a Dehn twist along $A$. Notice first that $j_n=D_{A_0}^n$ takes
a $(1,0)$-curve on $F$ to a $(1,n)$-curve on $F$, so extends to a homeomorphism
$j_n:M=\hat{M}(1,0)\to\hat{M}(1,n)$. Therefore, since $i_0\circ h_0$ and $i_0\circ h_1$
are homotopic, so are $j_n\circ i_0\circ h_0$ and $j_n\circ i_0\circ h_1$. But,
$j_n\circ i_0\circ h_0$ is homotopic to $i_n\circ h_0\circ D_A^n$ and $j_n\circ i_0\circ h_1=i_n\circ h_1$, which completes the proof that 
$i_n\circ h_1$ is homotopic to $i_n\circ h_0\circ D_A^n$ for all $n$.

Let $\rho_0 = (\psi \circ h_0)_*$ and $\rho_1 = (\psi \circ h_1)_*$.
Since $(h_i)_*$ induces an injection of $\pi_1(M)$ into $\pi_1(\hat M)$,
$\rho_i\in AH(M)$. We next observe that one can choose 
$\hat N$ so that
$N_{\rho_0}$ and $N_{\rho_1}$ are not isometric. 
Let $a_i=A_i\cap(\partial M- \partial_T \hat{M})$ and let $a_i^*$ denote the geodesic representative
of $\psi(a_i)$ in $\partial_c\hat N$.
Notice that for each $i=0,1$ there
is a conformal embedding of $\partial_c\hat{N}-a_i^*$ into $\partial_cN_{\rho_i}$ such that
each component of the complement  of the image of $\partial_c\hat{N}-a_i^*$ is a neighborhood
of a cusp. One may
therefore choose the conformal structure on $\partial_c\hat{N}$ so that
 there is not a conformal homeomorphism from
$\partial_cN_{\rho_0}$  to $\partial_cN_{\rho_1}$.
Therefore, $N_{\rho_0}$ and $N_{\rho_1}$ are not isometric.

Let $\{N_n=N_{\beta_n}\}$ be the sequence of hyperbolic 3-manifolds provided by the
Hyperbolic Dehn Filling Theorem applied to the sequence $\{(1,n)\}_{n\in\IZ_+}$ and
let $\{\psi_n: \hat M(1,n)-\partial_T\hat M(1,n)\to N_n\cup\partial_cN_n\}$ be
the homeomorphisms  such that $\psi_n\circ i_n\circ \psi^{-1}$ is conformal
on $\partial_cN$.
Let 
$$\rho_{n,i}=\beta_n\circ\rho_i$$
for all  $n$ large enough that $N_n$ and $\psi_n$ exist.
Since $\beta_n\circ \psi_*$ is conjugate to $(\psi_n\circ i_n)_*$ 
(by applying part (2) of the Hyperbolic Dehn Filling Theorem) and
$i_n\circ h_i$ is homotopic to a homeomorphism, we
see that $\rho_{n,i}=(\psi_n\circ i_n\circ h_i)_*$ lies in $AH(M)$ for all  $n$ and each $i$.
It follows from part (1) of the Hyperbolic Dehn Filling Theorem 
that $\{\rho_{n,i}\}$ converges to $\rho_i$ for each $i$.
Moreover, $\rho_{n,1} = \rho_{n,0} \circ (D_A)_*^n$ for all $n$,
since  $i_n\circ h_1$ is homotopic to $i_n\circ h_0\circ D_A^n$ for all $n$.
Therefore, ${\rm Out}(\pi_1(M))$ does not act properly discontinuously on $AH(M)$.

Moreover, $\{\rho_{n,0}\}$ and $\{\rho_{n,1}\}$ project to the same sequence in $AI(M)$
and both $N_{\rho_0}$ and $N_{\rho_1}$ are limits of this sequence.
Since $N_{\rho_0}$ and $N_{\rho_1}$ are distinct manifolds in $AI(M)$,
it follows that $AI(M)$ is not Hausdorff.
\end{proof}

\noindent {\bf Remark:}
One can also establish Theorem \ref{prop:!T2} using deformation
theory of Kleinian groups and convergence results of Thurston 
\cite{thurston3}.  This version of the argument follows the same
outline as the proof of Proposition 3.3 in \cite{CS}.

We provide a brief sketch of this argument.
The classical deformation theory
of Kleinian groups (in combination with Thurston's Hyperbolization
Theorem) guarantees that there exists a component $B$ of ${\rm int}(AH(M))$
such that if $\rho\in B$, then there exists a homeomorphism
$\bar h_\rho:M-\partial_TM\to N_\rho\cup\partial_c N_\rho$ and the point
$\rho$ is  determined by the induced conformal structure on $\partial M-\partial_TM$.
Moreover, every possible
conformal structure on $\partial M-\partial_TM$ arises in this manner. 

Let $a_0$ and $a_1$ denote the components of $\partial A$ and let
$t_{a_0}$ and $t_{a_1}$ denote Dehn twists about $a_0$ and $a_1$
respectively. We choose orientations so that $D_A$ induces
$t_{a_0}\circ t_{a_1}$ on $\partial M$. We then let $\rho_{n,0}\in B$
have associated conformal structure $t_{a_1}^n(X)$ and let
$\rho_{n,1}$ have associated conformal structure $t_{a_0}^{-n}(X)$
for some conformal structure $X$ on $\partial M$.
Thurston's convergence results \cite{thurston2,thurston3} can be used to show
that there exists a subsequence $\{ n_j\}$
of $\IZ$ such that $\{\rho_{n_j,0}\}$ and $\{\rho_{n_j,1}\}$ both converge. One
can guarantee, roughly as above, that the limiting hyperbolic manifolds are
not isometric. Moreover, $\rho_{n,1}=\rho_{n,0}\circ (D_A)_*^n$ for all $n$,
so we are the same situation as in the proof above.

\section{The characteristic submanifold and mapping class groups}\label{sec:mod}

In order to further analyze the case where $M$ has incompressible boundary
we will make use of the characteristic submanifold (developed by Jaco-Shalen \cite{JS} and Johannson \cite{johannson}) and the theory of mapping class groups
of 3-manifolds
developed by Johannson \cite{johannson} and extended by
McCullough and his co-authors \cite{VGF,UGF,CM}.

We begin by recalling the definition of the characteristic submanifold,
specialized to the hyperbolic setting. In the general setting, the components
of the characteristic submanifold are interval bundles and Seifert fibred spaces.
In the hyperbolic setting, the only Seifert fibred spaces which occur are
the solid torus and the thickened torus
(see Morgan  \cite[Sec. 11]{Morgan} or Canary-McCullough
\cite[Chap. 5]{CM}).

\begin{theorem}
\label{charprop}
Let $M$ be a compact oriented hyperbolizable $3$-manifold with incompressible boundary.  There exists a codimension zero submanifold
\hbox{$\Sigma(M) \subseteq M$} with frontier $Fr(\Sigma(M)) = \overline{\partial \Sigma(M) - \partial M}$ satisfying the following properties: 
\begin{enumerate}
\item Each component $\Sigma_i$ of $\Sigma(M)$ is
either 
\subitem(i)
an interval bundle  over a compact surface with negative
Euler characteristic which intersects $\partial M$ in its associated $\partial I$-bundle,
\subitem(ii)  a thickened torus such that  $\partial M\cap \Sigma_i$ contains a 
torus, or
\subitem(iii) a solid torus.
\item The frontier $Fr(\Sigma(M))$ is a collection of essential annuli.
\item Any essential annulus or incompressible torus in $M$
is properly isotopic  into $\Sigma(M)$.
\item
If $X$ is a component of $M-\Sigma(M)$, then either $\pi_1(X)$ is non-abelian
or $(\overline{X}, Fr(X))\cong (S^1\times [0,1]\times [0,1], S^1\times [0,1]\times \{0,1\}) $
and $X$ lies between an interval bundle component of
$\Sigma(M)$ and a thickened or solid torus component of $\Sigma(M)$.
\end{enumerate}
\noindent Moreover, such a $\Sigma(M)$ is unique up to isotopy, and is called the 
{\em characteristic submanifold} of $M$.
\end{theorem}

The existence and the uniqueness of the characteristic submanifold in general follows from The
Characteristic Pair Theorem in \cite{JS} or Proposition 9.4 and Corollary 10.9 
in \cite{johannson}.
Theorem \ref{charprop}(1) follows from \cite[Theorem 5.3.4]{CM}, part (2) follows from (1)
and the definition of the characteristic submanifold, part (3) follows
from \cite[Theorem 12.5]{johannson},  and part (4) follows from \cite[Theorem 2.9.3]{CM}.

Johannson's Classification Theorem \cite{johannson} asserts
that every homotopy equivalence  between compact, irreducible
3-manifolds with incompressible boundary
may be homotoped so that it
preserves the characteristic submanifold and is a homeomorphism on
its complement.  Therefore, the study of ${\rm Out}(\pi_1(M))$ often
reduces to the study of mapping class groups of interval bundles and
Seifert-fibered spaces.

\medskip\noindent
{\bf Johannson's Classification Theorem \cite[Theorem 24.2]{johannson}.}
{\em Let $M$ and $Q$ be irreducible 3-manifolds with incompressible boundary
and let $h:M\to Q$ be a homotopy equivalence. Then $h$ is homotopic to a map
$g:M\to Q$ such that
\begin{enumerate}
\item
$g^{-1}(\Sigma(Q))=\Sigma(M)$,
\item
$g|_{\Sigma(M)}:\Sigma(M)\to \Sigma(Q)$ is a homotopy equivalence,
and
\item
$g|_{\overline{M-\Sigma(M)}}:\overline{M-\Sigma(M)}\to \overline{Q-\Sigma(Q)}$ is
a homeomorphism.
\end{enumerate}
Moreover, if $h$ is a homeomorphism, then $g$ is a homeomorphim.
}

\medskip

We let the mapping class group ${\rm Mod}(M)$ denote the group of isotopy classes of  
self-homeomorphisms of $M$. Since $M$ is irreducible and has  (non-empty) incompressible boundary,  any two homotopic homeomorphisms are isotopic
(see Waldhausen \cite[Theorem 7.1]{Wald}),
so ${\rm Mod}(M)$ is naturally a subgroup of ${\rm Out}(\pi_1(M))$. 
For simplicity,
we will assume that $M$ is a compact hyperbolizable 3-manifold with incompressible
boundary and no toroidal boundary components. 
Notice that this implies that $\Sigma(M)$
contains no thickened torus components.
Let $\Sigma$ be the characteristic submanifold of $M$ and denote its components
by $\{\Sigma_1,\ldots,\Sigma_k\}$.
 
Following McCullough \cite{VGF},
we let ${\rm Mod}(\Sigma_i,Fr(\Sigma_i))$ denote the group of homotopy classes
of homeomorphisms  $h:\Sigma_i\to\Sigma_i$ such that $h(F)=F$ for each
component $F$ of $Fr(\Sigma_i)$. We let $G(\Sigma_i,Fr(\Sigma_i))$ denote
the subgroup consisting of (homotopy classes of) 
homeomorphisms which have representatives which are the identity on $Fr(\Sigma_i)$.
Define
$G(\Sigma, Fr(\Sigma))=\oplus_{i=1}^k G(\Sigma_i,Fr(\Sigma_i))$.  Notice that 
using these definitions, the restriction of  a Dehn twist along a component of $Fr(\Sigma)$ is trivial in $G(\Sigma, Fr(\Sigma))$. 
 
In our case, each $\Sigma_i$ is either an interval bundle over a compact surface $F_i$
with negative Euler characteristic or a solid torus.
If $\Sigma_i$ is a solid torus, then $G(\Sigma_i,Fr(\Sigma_i))$ is finite
(see Lemma 10.3.2 in \cite{CM}). 
If $\Sigma_i$ is an interval bundle over a compact surface $F_i$, then  $G(\Sigma_i,Fr(\Sigma_i))$ is isomorphic to  the group $G(F_i,\partial F_i)$ of
proper isotopy classes of self-homeomorphisms of $F$ which are the identity
on $\partial F$ (see Proposition 3.2.1 in \cite{VGF} and Lemma 6.1 in \cite{UGF}). 
Moreover, $G(\Sigma_i,Fr(\Sigma_i))$ injects into ${\rm Out}(\pi_1(\Sigma_i))$
(see Proposition 5.2.3 in \cite{CM} for example).
We say that $\Sigma_i$ is {\em tiny} if its base surface $F_i$ is either a thrice-punctured
sphere or a twice-punctured projective plane. If $\Sigma_i$ is not tiny, then
$F_i$ contains a 2-sided, non-peripheral homotopically non-trivial simple
closed curve, so $G(\Sigma_i,Fr(\Sigma_i))$ is infinite. If $\Sigma_i$ is tiny,
then $G(\Sigma_i,Fr(\Sigma_i))$ is finite (see Korkmaz \cite{korkmaz} for
the case when $F_i$ is a twice-punctured projective plane).

Let $J(M)$ be the subgroup of ${\rm Mod}(M)$ consisting
of classes represented by homeomorphisms fixing $M-\Sigma$ pointwise.
Lemma 4.2.1 of McCullough \cite{VGF} implies
that $J(M)$ has finite index in ${\rm Mod}(M)$. 
(Instead of $J(M)$, McCullough writes
$\mathcal{K}(M, \Sigma_1, \Sigma_2, \ldots, \Sigma_k)$.)  
Lemma 4.2.2 of McCullough
\cite{VGF} implies that the kernel $K(M)$ of the natural surjective homomorphism
$$p_\Sigma: J(M) \to G(\Sigma, Fr(\Sigma))$$ 
is abelian and is generated by Dehn twists about the annuli in $Fr(\Sigma)$.

We summarize the discussion above in the following statement.

\begin{theorem}
\label{Jstructure}
Let $M$ be a compact hyperbolizable 3-manifold with incompressible
boundary and no toroidal boundary components. Then there is
a finite index subgroup $J(M)$ of ${\rm Mod}(M)$ and an exact sequence
$$1\longrightarrow K(M) \longrightarrow J(M)\overset{p_\Sigma}\longrightarrow G(\Sigma, Fr(\Sigma))
\longrightarrow 1$$ 
such that $K(M)$ is an abelian group generated by Dehn twists about
essential annuli in $Fr(\Sigma)$. 

Suppose that $\Sigma_i$ is a component of $\Sigma(M)$.
If $\Sigma_i$ is a solid torus or a tiny
interval bundle, then $G(\Sigma_i,Fr(\Sigma_i))$ is finite. Otherwise,
$G(\Sigma_i,Fr(\Sigma_i))$ is infinite and injects into ${\rm Out}(\pi_1(\Sigma_i))$.
\end{theorem}

\section{Characteristic collections of annuli}

We continue to assume that $M$ has incompressible boundary
and no toroidal boundary components and that $\Sigma(M)$
is its characteristic submanifold.  In this section, we organize $K(M)$ into
subgroups generated by collections of annuli with homotopic core curves,
called characteristic collection of annuli, 
and define a class of free subgroups of $\pi_1(M)$ which ``register''
these subgroups of $K(M)$.

A {\em characteristic collection of annuli} for $M$ is either a) the collection
of all frontier annuli in a solid torus component of $\Sigma(M)$, or b)
an annulus in the frontier of an interval bundle component of $\Sigma(M)$ which
is not properly isotopic to a frontier annulus of a solid torus component of
$\Sigma(M)$. 

If $C_j$ is a characteristic collection of annuli for $M$,
let $K_j$ be the subgroup of $K(M)$ generated by Dehn twists
about the annuli in $C_j$. Notice that $K_i\cap K_j=\{ id\}$ for $i\ne j$,
since each element of $K_j$ fixes any curve disjoint from $C_j$.
Then $K(M)=\oplus_{j=1}^mK_j$, since every frontier
annulus of $\Sigma(M)$ is properly isotopic to a component of some
characteristic collection of annuli.
Let $q_j:K(M)\to K_j$ be the projection map.  

We next introduce free subgroups of $\pi_1(M)$, called $C_j$-registering subgroups, which are preserved by
$K_j$ and such that $K_j$ acts effectively on the subgroup.

We first suppose that $C_j=Fr(T_j)$ where $T_j$ is 
a solid torus component of $\Sigma(M)$.  Let $\{ A_1,\ldots, A_l\}$
denote the components of $Fr(T_j)$.  
For each $i=1,\ldots,l$, let $X_i$ be
the component of $M-(T_j \cup C_1 \cup C_2 \cup \ldots \cup C_m)$
abutting $A_i$. (Notice that each $X_i$ is either a component of $M-\Sigma(M)$ or properly
isotopic to the interior of an interval bundle component of $\Sigma(M)$.)
Let $a$ be a core curve for $T_j$ and let $x_0$ be a point on $a$.
We say that a
subgroup  $H$ of $\pi_1(M,x_0)$ is  {\em $C_j$-registering}
if it is freely (and minimally) generated by  $a$
and, for each $i=1,\ldots,l$, a loop $g_i$ in $T_j\cup X_i$ based at $x_0$ intersecting $A_i$ exactly twice.  In particular,
every $C_j$-registering subgroup of $\pi_1(M,x_0)$ is isomorphic to $F_{l+1}$.

Notice that a Dehn twist $D_{A_i}$ along any $A_i$ preserves $H$ in $\pi_1(M,x_0)$.
It acts on $H$ by the map $t_i$ which fixes $a$ and $g_m$ for $m \neq i$,
and conjugates $g_i$ by $a^n$ (where the core curve of $A_i$ is homotopic
to $a^n$). Let $s_H:K_j\to {\rm Out}(H)$ be
the homomorphism which takes each $D_{A_i}$ to $t_i$.
Simultaneously twisting along all $l$ annuli induces conjugation by $a^n$, which is an inner automorphism of $H$. Moreover, it is easily checked that $s_H(K_j)$ is isomorphic
to ${\bf Z}^{l-1}$ and is generated
by $t_1,\ldots,t_{l-1}$.
The set $\{a,g_1,\ldots,g_l\}$ may be extended to a generating
set for $\pi_1(M,x_0)$ by appending curves which intersect $Fr(T_j)$ exactly
twice, so $D_{A_1}\circ\cdots D_{A_l}$ acts as conjugation by $a^n$ on all of $\pi_1(M,x_0)$. Therefore,
$K_j$ itself is isomorphic to ${\bf Z}^{l-1}$ and $s_H$ is
injective. (In particular, if $C_j$ is a single annulus in the boundary of a solid torus
component of $\Sigma(M)$, then $K_j$ is trivial and we could have omitted $C_j$.) 

Now suppose that $C_j=\{A\}$ is a frontier annulus of an interval 
bundle component $\Sigma_i$ of
$\Sigma$ which is not properly isotopic into a solid torus component 
of $\Sigma$.  Let $a$ be a core curve for $A$ and let $x_0$ be
a point on $a$. We say that a subgroup $H$ of $\pi_1(M,x_0)$ is
{\em $C_j$-registering} if it is freely (and minimally)
generated by $a$ and two loops $g_1$ and $g_2$ based
at $x_0$ each of whose interiors misses $A$, and which lie in the two
distinct components of $M-(C_1 \cup C_2 \cup \ldots \cup C_m)$ abutting $A$.  In this case, $H$ is isomorphic to $F_3$.  Arguing as above, it follows that $K_j$ is an infinite cyclic subgroup of ${\rm Out}(\pi_1(M))$ and there is an injection  $s_H:K_j\to {\rm Out}(H)$. 

In either situation, if $H$ is a $C_j$-registering group for a characteristic collection of
annuli $C_j$, then we may consider the map 
$$r_H:X(M)\to X(H)$$
simply obtained
by taking $\rho$ to $\rho|_H$.  (Here, $X(H)$ is the ${\rm PSL}_2({\bf C})$-character variety of the abstract
group $H$.)
One easily checks from the description above
that if $\alpha\in K_j$, then $r_H(\rho\circ\alpha)=r_H(\rho)\circ s_H(\alpha)$ for
all $\rho\in X(M)$.
Notice that if $\phi\in K_l$ and
$j\ne l$, then $K_l$ acts trivially on $H$, since each generating curve of $H$
is disjoint from $C_l$. Therefore, $$r_H(\rho\circ\alpha)=r_H(\rho)\circ s_H(q_j(\alpha))$$ for
all $\rho\in X(M)$ and $\alpha\in K(M)$.

We summarize the key points of this discussion for use later:

\begin{lemma}
\label{Kjinject}
Let $M$ be a compact hyperbolizable 3-manifold with incompressible
boundary and no toroidal boundary components.
If $C_j$ is a characteristic collection of annuli for $M$ and
$H$ is a $C_j$-registering subgroup of $\pi_1(M)$,
then $H$ is preserved by each element of $K_j$ and there is a natural
injective homomorphism $s_H:K_j\to {\rm Out}(H)$. Moreover,  if $\alpha\in K(M)$,
then $r_H(\rho\circ\alpha)=r_H(\rho)\circ s_H(q_j(\alpha))$ for all $\rho\in X(M)$.
\end{lemma}

\section{Primitive essential annuli and manifolds with compressible boundary}

In this section we use a result of Johannson \cite{johannson} to show that every compact
hyperbolizable 3-manifolds with compressible boundary and no toroidal boundary
components contains a primitive essential annulus. 
It then follows from Theorem \ref{prop:!T2} that if $M$ has compressible
boundary and no toroidal boundary components, then ${\rm Out}(\pi_1(M))$ fails to
act properly discontinuously on $AH(M)$ and  $AI(M)$ is not
Hausdorff.

We first find indivisible curves in the boundary of compact hyperbolizable
3-manifolds with incompressible boundary and no toroidal boundary components.
We call a curve $a$ in $M$ {\em indivisible} if it generates a maximal
cyclic subgroup of $\pi_1(M)$.

\begin{lemma}
\label{primexists} 
Let $M$ be a compact hyperbolizable 3-manifold with (non-empty) incompressible
boundary. Then, if $F$ is a component of $\partial M$,
there exists an indivisible simple closed curve in $F$.
\end{lemma}

\begin{proof}{} 
We use a special
case of a result of Johannson \cite{johannson} (see also Jaco-Shalen \cite{JS-peripheral})
which characterizes divisible simple closed curves in $\partial M$.

\begin{lemma}{\rm (\cite[Lemma 32.1]{johannson})}
\label{divisible}
Let $M$ be a compact hyperbolizable  3-manifold with incompressible boundary.
An essential simple closed curve $\alpha$
in $\partial M$  which is not indivisible
is  either isotopic into a solid torus component of $\Sigma(M)$ or
is isotopic to a boundary component of an essential M\"obius band in
an interval bundle component of $\Sigma(M)$.
\end{lemma}

Therefore, if $\Sigma(M)$ is not all of $M$, then any simple closed curve in $F$
which cannot be isotoped into a solid torus or interval bundle component of $\Sigma(M)$ is indivisible.

If $\Sigma(M)=M$, then $M$ is an interval bundle over a closed surface
with negative Euler characteristic and the proof is completed by the following lemma,
whose full statement will be used later in the paper.

\begin{lemma}{}
\label{primannuli}
Let $M$ be a compact hyperbolizable 3-manifold
with no toroidal boundary components.
Let $\Sigma_i$ be an interval bundle component of $\Sigma(M)$ which is not tiny,
then there is a primitive essential annulus (for $M$) contained in $\Sigma_i$.
\end{lemma}

\begin{proof}
Let $F_i$ be the base surface of $\Sigma_i$. Since $\Sigma_i$ is not tiny,
$F_i$ contains a non-peripheral
simple closed curve $a$ which is two-sided and homotopically non-trivial.
Then $a$ is an indivisible curve in $F_i$ and hence in $M$. 
The sub-interval bundle $A$ over $a$ is thus a primitive essential annulus.
\end{proof}
\end{proof}

We are now prepared to prove the main result of the section.

\begin{prop}\label{lem:comp_implies_annulus}
If $M$ is a compact hyperbolizable $3$-manifold with compressible boundary and no toroidal boundary components, then $M$ contains a primitive essential annulus.
\end{prop}

\begin{proof}
We first observe that under our assumptions every
maximal abelian subgroup of $\pi_1(M)$ is cyclic (since every non-cyclic
abelian subgroup of the fundamental group of a compact hyperbolizable
3-manifold is conjugate into the fundamental group of a toroidal
component of $\partial M$, see \cite[Corollary 6.10]{Morgan}). Therefore, in
our case an essential annulus is primitive if and only if its core curve
is indivisible.

We first suppose that $M$ is a compression body. If $M$ is a handlebody, then
it is an interval bundle, so contains a primitive essential annulus by Lemma
\ref{primannuli}.
Otherwise, $M$ is formed from $R\times I$
by appending 1-handles to $R\times \{1\}$, where $R$ is a closed, but not necessarily
connected, orientable surface. Let $\alpha$ be an essential
simple closed curve in
$R\times \{1\}$ which lies in $\partial M$.  Let $D$ be a
disk in $R\times \{1\}-\partial M$. We may assume that $\alpha$ intersects $\partial D$ in exactly one point.  Let $\beta \subset (\partial M\cap R\times\{1\})$ be a simple closed curve homotopic
to $\alpha \ast \partial D$ (in $\partial M$) and disjoint from $\alpha$.
Then $\alpha$ and $\beta$ bound an embedded annulus in $R\times \{1\}$,
which may be homotoped to a primitive essential annulus in $M$ (by pushing
the interior of the annulus into the interior of $R\times I$).

If $M$ is not a compression body, 
let $C_M$ be a characteristic compression body  neighborhood of $\partial M$
(as discussed in Section \ref{sec:prelim}).
Let $C$ be a component of $C_M$ which has a compressible
boundary component $\partial_+C$ and an incompressible boundary component $F$.
Let $X$ be the component of $\overline{M-C_M}$ which 
contains $F$ in  its boundary and let $\alpha$ be an essential simple closed curve in
$F$ which is indivisible in $X$ (which exists by Lemma \ref{primexists}).
Let $\alpha'$ be a curve in $\partial_+C\subset\partial M$ which is homotopic to 
$\alpha$.
One may then construct as above a primitive essential annulus $A$  in $C$
with $\alpha'$ as one boundary component.
It is clear that $A$ remains essential in $M$.
Since $\pi_1(M)=\pi_1(X)*H$ for some group $H$, the core curve of $A$,
which is homotopic to $\alpha$, is indivisible in $\pi_1(M)$.
Therefore, $A$ is our desired primitive essential annulus in $M$.
\end{proof}

\noindent {\bf Remark:}
The above argument is easily extended to the case where $M$ is allowed
to have toroidal boundary components (but is still hyperbolizable), unless
$M$ is a compression body all of whose boundary components are tori. In fact,
the only counterexamples in this situation occur when $M$ is obtained from
one or two untwisted interval bundles over tori by attaching exactly one 1-handle.

\medskip

We have thus already established Corollary \ref{thm:intro2} in the case that $M$ has
compressible boundary.

\begin{corollary}\label{cor:compressible_not_T2}
If $M$ is a compact hyperbolizable $3$-manifold with compressible boundary,
no toroidal boundary components, and non-abelian fundamental group, 
then  ${\rm Out}(\pi_1(M))$ does not act properly discontinuously on $AH(M)$.
Moreover, the moduli space $AI(M)$ is not Hausdorff.
\end{corollary}

\section{The space $AH_n(M)$}

In this section, 
we assume that $M$ has incompressible boundary and
no toroidal boundary components. We identify a subset $AH_n(M)$ of
$AH(M)$ which contains all purely hyperbolic representations in $AH(M)$.
We will see later that ${\rm Out}(\pi_1(M))$ acts
properly discontinuously on  an open neighborhood of $AH_n(M)$ in $X(M)$
if $M$ is not an interval bundle.

We define $AH_n(M)$ to be the set of (conjugacy classes of)
representations $\rho\in AH(M)$ such that 
\begin{enumerate}
\item
If $\Sigma_i$ is a component of the characteristic submanifold which
is not a tiny interval bundle, then $\rho(\pi_1(\Sigma_i))$ is purely
hyperbolic (i.e. if $g$ is a non-trivial element of $\pi_1(M)$ which is conjugate into
$\pi_1(\Sigma_i)$, then
$\rho(g)$ is hyperbolic), and
\item
if $\Sigma_i$ is a tiny interval bundle, then $\rho(\pi_1(Fr(\Sigma_i))$ is
purely hyperbolic.
\end{enumerate}

We observe that ${\rm int}(AH(M))$ is a proper subset of $AH_n(M)$ and
that $AH(M)=AH_n(M)$ if and only if $M$ contains no primitive essential annuli.

\begin{lemma}
\label{allofit}
Let $M$ be a compact hyperbolizable 3-manifold with non-empty
incompressible boundary and no toroidal boundary components.  Then
\begin{enumerate}
\item
the interior of $AH(M)$ is  a proper subset of $AH_n(M)$, 
\item $AH_n(M)$ contains a dense subset of $\partial AH(M)$,
and
\item
$AH_n(M)=AH(M)$ if and only if $M$ contains no
primitive essential annuli.
\end{enumerate}
\end{lemma}

\begin{proof} 
Sullivan \cite{sullivan2} proved that all representations in ${\rm int}(AH(M))$
are purely hyperbolic (if $M$ has no toroidal boundary components), so
clearly ${\rm int}(AH(M))$ is contained in $AH_n(M)$. On the other hand,
$\partial AH(M)$ is non-empty  (see  Lemma 4.1 in Canary-Hersonsky \cite{CH})
and purely hyperbolic representations are dense in $\partial AH(M)$ (which follows from
Lemma 4.2 in \cite{CH} and the Density Theorem \cite{brock-bromberg,BCM,HN,ohshika-density}). This establishes claims (1) and (2).

If $M$ contains a primitive essential annulus $A$, then there exist $\rho\in AH(M)$
such that $\rho(\alpha)$ is parabolic (where $\alpha$ is the core curve of $A$),
so $AH_n(M)$ is not all of $AH(M)$  in this case (see Ohshika \cite{ohshika-parabolic}).

Now suppose that $M$ contains no primitive essential annuli.
We first note that every component of $\Sigma(M)$ is a solid torus or
tiny interval bundle (by Lemma \ref{primannuli}). Moreover, if $\Sigma_i$ is
a tiny interval bundle component of $\Sigma(M)$, then any component $A$
of its frontier must be isotopic to a component of the frontier of
a solid torus component of $\Sigma(M)$. Otherwise, $A$ would be a primitive
essential annulus (by Lemma \ref{divisible}).
Therefore, it suffices to prove that $\rho(\Sigma_i)$ is purely hyperbolic
whenever $\Sigma_i$ is a solid torus component of $\Sigma(M)$.

Let $T$ be a solid torus component of $\Sigma(M)$.
A frontier annulus $A$ of $T$ is an essential annulus in $M$, so it must not be primitive.  It follows that the core curve $a$ of $T$ is not peripheral in $M$
(see \cite[Theorem 32.1]{johannson}).
 
Let $\rho\in AH(M)$ and let $R$ be a relative compact core for
$(N_\rho)^0_\epsilon$ (for some $\epsilon<\mu).$
Let $h:M\to R$ be a homotopy equivalence in the homotopy class determined by $\rho$.
By Johannson's Classification Theorem \cite[Thm.24.2]{johannson}, $h$ may
be homotoped so that $h(T)$ is a component $T'$ of $\Sigma(R)$,
$h|_{Fr(T)}$  is an embedding with image $Fr(T')$
and  $h|_T:(T,Fr(T))\to (T',Fr(T'))$ is a homotopy equivalence of pairs.
It follows that $h(a)$ is homotopic to the core curve of $T'$ which is
not peripheral in $R$.

If $\rho(a)$ were parabolic, then $h(a)$  would be homotopic into a non-compact
component of $(N_\rho)_{thin(\epsilon)}$ and hence into
$P=R\cap \partial( N_\rho)^0_\epsilon\subset\partial R$, so $h(a)$ would be peripheral 
in $R$. It follows that $\rho(a)$ is hyperbolic. Since $a$ generates $\pi_1(T)$, we see
that $\rho(\pi_1(T))$ is purely hyperbolic.
Since $T$ is an arbitrary solid torus component of $\Sigma(M)$,
we see that $\rho\in AH_n(M)$.
\end{proof}

We next check that the restriction of $\rho\in AH_n(M)$ to the fundamental group of an
interval bundle component of $\Sigma(M)$ (which is not tiny) is Schottky.
By definition, a Schottky group is a free, geometrically finite, purely hyperbolic subgroup of $\PSL_2({\bf C})$ (see Maskit \cite{maskit-free} for
a discussion of the equivalence of this definition with more classical definitions).

\begin{lemma} 
\label{IbundleSchottky}
Let $M$ be a compact hyperbolizable 3-manifold with incompressible
boundary with no toroidal boundary components which is not an interval bundle.
If $\Sigma_i$ is an  interval bundle component of $\Sigma(M)$ which is not tiny and
\hbox{$\rho\in AH_n(M)$}, then $\rho(\pi_1(\Sigma_i))$ is a Schottky group.
\end{lemma}

\begin{proof}
By definition $\rho(\pi_1(\Sigma_i))$ is purely hyperbolic,
so it suffices to prove it is free and geometrically finite.
Since $\Sigma_i$ is an interval bundle whose base surface $F_i$ has
non-empty boundary, $\pi_1(\Sigma_i)\cong\pi_1(F_i)$ is free.
Let $\pi_i:N_i\to N_\rho$ be the cover of $N_\rho$ associated to 
$\rho(\pi_1(\Sigma_i))$. 
Since $\pi_1(\Sigma_i)$ has infinite index in $\pi_1(M)$,
$\pi_i:N_i\to N$ is a covering with infinite degree.
Let $R_i$ be a compact core for $N_i$. Since $\pi_1(R_i)$ is free and $R_i$ is irreducible,
$R_i$ is a handlebody (\cite[Theorem 5.2]{hempel}).
Therefore, $N_i=(N_i)^0_\epsilon$ has
one end and $\pi_i$ is infinite-to-one on this end, so the Covering Theorem 
(see \cite{cover}) implies that this end is geometrically finite, and hence that
$N_i$ is geometrically finite. Therefore, $\rho(\pi_1(\Sigma_i))$ is geometrically finite,
completing the proof that it is a Schottky group.
\end{proof}

Finally, we check that if $\rho\in AH_n(M)$  and $C_j$ is a characteristic
collection of annuli, then there exists a $C_j$-registering subgroup
whose image under $\rho$ is Schottky.

\begin{lemma}
\label{CjSchottky}
Suppose that $M$ is a compact hyperbolizable 3-manifold with incompressible
boundary and no toroidal boundary components
and $C_j$ is a characteristic collection of frontier annuli for $M$.
If $\rho\in AH_n(M)$, then there exists a $C_j$-registering subgroup $H$ of
$\pi_1(M)$ such that $\rho(H)$ is a Schottky group.
\end{lemma}

\begin{proof}
We first suppose that $C_j=\{ A\}$ is a frontier annulus of an interval bundle component of
$\Sigma(M)$ (and that $A$ is not properly isotopic to a frontier annulus of
a solid torus component of $\Sigma(M)$) and let $x_0\in A$. We identify $\pi_1(M)$ with $\pi_1(M,x_0)$.
Let $X_1$ and $X_2$ be the (distinct) components of \hbox{$M-Fr(\Sigma)$} abutting $A$.
Notice that each $X_i$ must have non-abelian fundamental group, since
it either contains (the interior of) an interval bundle component of $\Sigma(M)$ or 
(the interior of) a component of $M-\Sigma(M)$ which is not a solid torus lying
between an interval bundle component of $\Sigma(M)$ and a solid torus component
of $\Sigma(M)$.

Let $a$ be the core curve of $A$ (based at $x_0$). By assumption,
$\rho(a)$ is a hyperbolic element.
Let $F$ be a fundamental
domain for the action of \hbox{$<\rho(a)>$} on $\Omega(<\rho(a)>)$
which is an annulus in $\rs$. Since each $\rho(\pi_1(\overline{X_i},x_0))$ is discrete,
torsion-free and non-abelian, hence non-elementary, we may
choose hyperbolic elements $\gamma_i\in \rho(\pi_1(\overline{X_i},x_0))$ 
whose fixed points lie in the interior of $F$. There exists $s>0$ such that one may
choose (round) disks $D_i^\pm\subset {\rm int}(F)$ about the fixed points 
of $\gamma_i$, such that $\gamma_i^s({\rm int}(D_i^-))=\rs -D_i^+$, and 
$D_1^+$, $D_1^-$, $D_2^+$ and $D_2^-$ are disjoint.  
Then, the Klein Combination Theorem (commonly referred to as
the ping pong lemma), guarantees that $\rho(a)$,
$\gamma_1^s$ and $\gamma_2^s$ freely generate a Schottky group,
see, for example, Theorem C.2 in Maskit \cite{maskit-book}.
Then each $\rho^{-1}(\gamma_i^s)$ is represented by a curve $g_i$ in $\overline{X_i}$
based at $x_0$ and $a$, $g_1$ and $g_2$ generate a $C_j$-registering
subgroup $H$ such that $\rho(H)$ is Schottky.

Now suppose that $C_j=\{A_1,\ldots,A_l\}$ is the collection of frontier annuli
of a solid torus component $T_j$ of $\Sigma(M)$. Let $X_i$ be the component
of \hbox{$M - (T_j \cup C_1 \cup \ldots \cup C_m)$} abutting $A_i$.
Pick $x_0$ in $T_j$ and let $a$ be a core curve of $T_j$ passing through $x_0$. 
Again each $X_i$ must have non-abelian fundamental group.

Let $F$ be an annular fundamental domain for the action of $<\rho(a)>$ on the complement
in $\rs$ of the fixed points of $\rho(a)$.
For each $i$, let \hbox{$Y_i=X_i\cup A_i\cup {\rm int}(T_j')$} and
pick a hyperbolic element $\gamma_i$ in $\rho(\pi_1(Y_i,x_0))$
both of whose fixed points lie in the interior of $F$. (Notice that even though
it could be the case that $X_i=X_k$ for $i\ne k$, we still have
that $\pi_1(Y_i,x_0)$ intersects $\pi_1 (Y_k,x_0)$ only in the subgroup
generated by $a$, so these hyperbolic elements are all distinct.) Then, just
as in the previous case, there exists $s>0$ such that the elements
$\{\rho(a),\gamma_1^s,\ldots,\gamma_l^s\}$ freely generate a Schottky group.
Each $\rho^{-1}(\gamma_i^s)$ can be represented by a loop $g_i$
based at $x_0$ which lies in $Y_i$ and intersects
$A_i$ exactly twice.  Therefore, the group $H$ generated by $\{a,g_1,\ldots,g_2\}$
is $C_j$-registering and $\rho(H)$ is Schottky.
\end{proof}

\section{Proper discontinuity on $AH_n(M)$}
\label{propdisc}

We are finally prepared to prove that ${\rm Out}(\pi_1(M))$ acts properly discontinuously
on an open neighborhood of $AH_n(M)$ if $M$ is a compact hyperbolizable
3-manifold with incompressible boundary and no toroidal boundary components
which is not an interval bundle.

\begin{theorem}
\label{Pdiscnbhd} Let $M$ be a compact hyperbolizable $3$-manifold with nonempty incompressible boundary and no toroidal boundary components which is not an
interval bundle. Then there exists an open ${\rm Out}(\pi_1(M)$-invariant neighborhood 
$W(M)$ of $AH_n(M)$ in $X(M)$ such that ${\rm Out}(\pi_1(M))$  acts properly discontinuously on $W(M)$.
\end{theorem}

\medskip

Notice that Theorem \ref{thm:intro3} is an immediate consequence of
Proposition \ref{lem:comp_implies_annulus},
Lemma \ref{allofit} and Theorem \ref{Pdiscnbhd}.
Moreover,
Theorem \ref{openpd} is an 
immediate corollary of Lemma \ref{allofit}
and Theorem \ref{Pdiscnbhd}.

\medskip

We now provide a brief outline of the section. In section 9.1 we recall Minsky's
work which shows that ${\rm Out}(\pi_1(H_n))$ acts properly discontinuously
on the open set $PS(H_n)$ of primitive-stable representations in $X(H_n)$
where $H_n$ is the handlebody of genus $g$.  In section 9.2,
we consider the set $Z(M)\subset X(M)$
such that if  $\rho\in Z(M)$ and $C_j$ is a characteristic collection of annuli,
then there exists a $C_j$-registering subgroup $H$ of $\pi_1(M)$ such
that $\rho|_H$ is primitive stable. We use Minsky's work to show that 
$K(M)$ acts properly discontinuously on $Z(M)$. In section 9.3, we consider the set
$V(M)$ of all representation such that $\rho|_{\pi_1(\Sigma_i)}$ is primitive-stable whenever
$\Sigma_i$ is an interval bundle component of $\Sigma(M)$ which is not tiny. 
We show that if $\{\alpha_n\}$ is a sequence in $J(M)$ such that $\{\rho_\Sigma(\alpha_n)\}$
is a sequence of distinct elements and $K$ is  compact subset of $V(M)$,
then $\{\alpha_n(K)\}$ leaves every compact set.
In section 9.4, we let $W(M)=Z(M)\cap V(M)$ and combine
the work in the previous sections to show that $J(M)$
acts properly discontinuously on $W(M)$. Since $J(M)$ has finite
index in ${\rm Out}(\pi_1(M))$ (see \cite{CM}),
this immediately implies Theorem \ref{Pdiscnbhd}. Johannson's Classification
Theorem is used to show that $J(M)$ is  invariant under ${\rm Out}(\pi_1(M))$.

\subsection{Schottky groups and primitive-stable groups}

In this section, we recall Minsky's work \cite{minsky-primitive} on
primitive-stable representations of the free group $F_n$, where $n\ge 2$.
An element of $F_n$ is called {\em primitive} if it is an element of a minimal free
generating set for $F_n$. Let $X$ be a bouquet of $n$ circles with base point $b$
and fix a specific identification of $\pi_1(X,b)$ with $F_n$. To a conjugacy class
$[w]$ in $F_n$ one can associated an infinite geodesic in $X$ which 
is obtained by concatenating  infinitely many copies of a cyclically reduced
representative of $w$ (here the cyclic reduction is in the generating set
associated to the natural generators of $\pi_1(X,b)$). Let $\mathcal{P}$
denote the set of infinite geodesics in the universal cover $\tilde X$ of
$X$ which project to geodesics associated to primitive words of $F_n$.

Given a representation $\rho:F_n\to \PSL_2({\bf C})$, $x\in\H^3$ and a lift
$\tilde b$ of $b$, one obtains a unique $\rho$-equivariant
map $\tau_{\rho,x}:\tilde X\to \H^3$ which
takes $\tilde b$ to $x$ and maps each edge of $\tilde X$ to a geodesic.
A representation $\rho:F_n\to \PSL_2({\bf C})$ is {\em primitive-stable} if there
are constants $K,\delta>0$ such that $\tau_{\rho,x}$ takes all
the geodesics in $\mathcal{P}$ to $(K,\delta)$-quasi-geodesics in $\H^3$.
We let $PS(H_n)$ denote the set of (conjugacy classes) of primitive-stable
representations  in $X(H_n)$ where $H_n$ is the handlebody of genus $n$.

We summarize the key points of  Minsky's work which we use
in the remainder of the section.
We recall that Schottky space $\mathcal{S}_n\subset X(H_n)$ is the space
of discrete faithful representations whose image is a Schottky group and
that $\mathcal{S}_n$
is the interior  of $AH(H_n)$.  

\begin{theorem}
\label{PSfacts}
{\rm (Minsky \cite{minsky-primitive})}
If $n\ge2$,  then
\begin{enumerate}
\item
${\rm Out}(F_n)$ acts properly discontinuously on $PS(H_n)$, 
\item
$PS(H_n)$ is an open subset of $X(H_n)$, and
\item
Schottky space $\mathcal{S}_n$ is a proper subset of $PS(H_n)$.
\end{enumerate}
Moreover, if $K$ is any compact subset of $PS(H_n)$,
and $\{\alpha_n\}$ is a sequence of distinct elements of ${\rm Out}(F_n)$,
then  $\{\alpha_n(K)\}$ exits every compact subset of $X(H_n)$ (i.e. for any
compact subset $C$ of $X(H_n)$ there exists $N$ such that if $n\ge N$,
then $\alpha_n(K)\cap C=\emptyset$).
\end{theorem}

\noindent
{\bf Remark:}
In order to prove our main theorem it would suffice to use
Schottky space $\mathcal{S}_n$ in place of $PS(H_n)$. However,
the subset $W(M)$ we obtain using $PS(H_n)$ is larger
than we would obtain using simply $\mathcal{S}_n$.

\subsection{Characteristic collection of annuli}

We will assume for the remainder of the section that $M$ is a
compact hyperbolizable 3-manifold with incompressible boundary and
no toroidal boundary components which is not an interval bundle.
Main Topological Theorem 2 in Canary and McCullough \cite{CM} 
(which is itself an exercise in applying Johannson's theory) implies that that
if $M$ has incompressible boundary and no toroidal boundary components,
then ${\rm Mod}(M)$ has finite index in ${\rm Out}(\pi_1(M))$. Therefore,
applying Theorem \ref{Jstructure}, we see that $J(M)$ has finite
index in ${\rm Out}(\pi_1(M))$.

Let $C_j$ be a characteristic collection of annuli in $M$.
If $H$ is a $C_j$-registering subgroup of $\pi_1(M)$, then 
the inclusion of $H$ in $\pi_1(M)$
induces a natural injection $s_H:K_j\to {\rm Out}(H)$ 
such that  if $\alpha\in K(M)$, then $r_H(\rho\circ\alpha)=r_H(\rho)\circ s_H(q_j(\alpha))$
where $r_H(\rho)=\rho|_H$ (see Lemma \ref{Kjinject}).
Let $Z_H=r_H^{-1}(PS(H))$ where $PS(H)\subset X(H)$ is the set of 
(conjugacy classes of) primitive-stable
representations of $H$.
Let $Z(C_j)=\bigcup Z_H$ where the union is taken over all
$C_j$-registering subgroups $H$ of $\pi_1(M)$.

If $\{ C_1,\ldots, C_m\}$ is the set of all characteristic collections of annuli for $M$,
then we define
$$Z(M)=\bigcap_{i=1}^m Z(C_j).$$
We use Lemma \ref{CjSchottky}, Theorem \ref{PSfacts}, and Johannson's Classification
Theorem to prove:

\begin{lemma}{}
\label{controlK} 
Let $M$ be a compact hyperbolizable $3$-manifold with nonempty incompressible boundary and no toroidal boundary components. Then
\begin{enumerate}
\item 
$Z(M)$ is an ${\rm Out}(\pi_1(M))$-invariant open neighborhood of $AH_n(M)$ in $X(M)$,
and
\item
if $K\subset Z(M)$ is compact and
$\{\alpha_n\}$ is a sequence of distinct elements of $K(M)$,
then $\alpha_n(K)$ exits every compact set of $X(M)$.
\end{enumerate}
\end{lemma}

\begin{proof}
Lemma \ref{CjSchottky} implies that $AH_n(M)\subset Z(C_j)$  for each $j$,
so $AH_n(M)\subset Z(H)$. Moreover, since $r_H$ is continuous for all $H$, each $Z(C_j)$
is open, and hence $Z(M)$ is open.

Johannson's Classification Theorem implies
that if $C_j$ is a characteristic collection of annuli for $M$ and $\phi\in {\rm Out}(\pi_1(M))$,
then there exists a homotopy equivalence $h:M\to M$ such that $h_*=\phi$ and
$h(C_j)$ is also a characteristic collection of annuli for $M$. Moreover, if $H$ is
a $C_j$-registering subgroup of $\pi_1(M)$, then $\phi(H)$ is
a $h(C_j)$-registering subgroup of $\pi_1(M)$.  Therefore, $Z(M)$ is
${\rm Out}(\pi_1(M))$-invariant, completing
the proof of  claim (1).

If (2) fails to hold, then there is a compact subset $K$ of $Z(M)$, a compact
subset $C$ of $X(M)$ and a sequence  $\{\alpha_n\}$ of distinct elements of $K(M)$
such that $\alpha_n(K)\cap C$ is non-empty for all $n$. We may pass to
a subsequence, still called $\{\alpha_n\}$, so that there exists $j$ such that $\{q_j(\alpha_n)\}$ is a sequence
of distinct elements.  Since $X(M)$ is locally compact, 
for each $x\in K$, there exists an open neighborhood $U_x$
of $x$ and a $C_j$-registering subgroup $H_x$ such that the closure
$\bar U_x$ is a compact subset of $Z_{H_x}$. Since $K$ is compact,
there exists a finite collection of points $\{ x_1,\ldots, x_r\}$ such that
$K\subset U_{x_1}\cup \cdots\cup U_{x_r}$. Therefore, again passing
to subsequence if necessary, there must exists
$x_i$ such that $\alpha_n(U_{x_i})\cap C$ is non-empty for all $n$. Let
$U'=U_{x_i}$ and $H'=H_{x_i}$.
Lemma \ref{Kjinject} implies that 
$\{s_{H'}(q_j(\alpha_n))\}$ is a sequence of distinct elements of ${\rm Out}(H')$
and that $s_{H'}(q_j(\alpha_n))(r_{H'}(\bar U'))=r_{H'}(\alpha_n(\bar U'))$.
Theorem \ref{PSfacts} then
implies that  $\{s_{H'}(q_j(\alpha_n))(r_{H'}(\bar U'))\}=\{ r_{H'}(\alpha_n(\bar U'))\}$ exits every compact
subset of $X(H')$. Therefore, $\{\alpha_n(U')\}$ exits every compact subset of
$X(M)$ which is  a contradiction. We have thus established (2).
\end{proof}

\subsection{Interval bundle components of $\Sigma(M)$}

Let $\Sigma_i$ be an interval bundle component of $\Sigma(M)$ with base
surface $F_i$ and let $X(\Sigma_i)$ be its associated character variety.
There exists a natural restriction map $r_i:X(M)\to X(\Sigma_i)$ taking 
$\rho$ to $\rho|_{\pi_1(\Sigma_i)}$.
Recall that $G(\Sigma_i,Fr(\Sigma_i))$ injects into ${\rm Out}(\pi_1(\Sigma_i))$
(by Lemma \ref{Jstructure}), 
so acts effectively on $X(\Sigma_i)$.
Moreover, if $\alpha\in J(M)$, then 
$r_i(\rho\circ \alpha)=r_i(\rho)\circ p_i(\alpha)$
where $p_i$ is the projection of $J(M)$ onto $G(\Sigma_i, Fr(\Sigma_i))$.
If $\Sigma_i$ is not tiny, 
we define
$$V(\Sigma_i)=r_i^{-1}(PS(\Sigma_i)).$$

If $\{\Sigma_1,\ldots,\Sigma_n\}$ denotes the collection of
all interval bundle components of $\Sigma(M)$ which are not tiny, then
we let
$$V(M)=\bigcap_{i=1}^n V(\Sigma_i).$$

We use Lemma \ref{IbundleSchottky}, Theorem \ref{PSfacts}, and Johannson's Classification
Theorem to prove:

\begin{lemma}
\label{controlMod} 
Let $M$ be a compact hyperbolizable $3$-manifold with nonempty incompressible boundary and no toroidal boundary components which is not an
interval bundle. Then
\begin{enumerate}
\item
$V(M)$ is an ${\rm Out}(\pi_1(M))$-invariant open neighborhood of $AH_n(M)$ in $X(M)$, and
\item
if $K$ is a compact subset of $V(M)$ and $\{\alpha_n\}$ is a sequence in $J(M)$
such that $\{p_\Sigma(\alpha_n)\}$ is
a sequence of distinct elements
of $G(\Sigma,Fr(\Sigma))$, then $\{\alpha_n(K)\}$ exits every compact subset of $X(M)$.
\end{enumerate}
\end{lemma}

\begin{proof}
Lemma \ref{IbundleSchottky}
implies that $AH_n(M)\subset V(\Sigma_i)$, for each $i$, and  each $V(\Sigma_i)$
is open since $r_i$ is continuous. Therefore, $V(M)$ is an open neighborhood of $AH_n(M)$.

Johannson's Classification Theorem implies that if $\phi\in {\rm Out}(\pi_1(M))$,
then there exists  a homotopy equivalence $h:M\to M$ such that $h(\Sigma(M))\subset\Sigma(M)$, $h|_{Fr(\Sigma)}$ is a self-homeomorphism of $Fr(\Sigma)$
and  $h$ induces $\phi$. Therefore, if $\Sigma_i$ is an interval bundle component
of $\Sigma(M)$, then \hbox{$\phi(\pi_1(\Sigma_i))$} is conjugate 
to $\pi_1(\Sigma_j)$ where
$\Sigma_j$ is also an interval bundle component of $\Sigma(M)$. Moreover,
if $\Sigma_i$ is not tiny, then $\pi_1(\Sigma_j)$ is also not tiny
(since $h|_{\Sigma_i}:\Sigma_i\to\Sigma_j$ is a homotopy equivalence which
is a homeomorphism on the frontier). 
It  follows that $V(M)$ is invariant under ${\rm Out}(\pi_1(M))$, completing
the proof of  claim (1).

If (2) fails to hold, then there is a compact subset $K$ of $Z(M)$, a compact
subset $C$ of $X(M)$ and a sequence  $\{\alpha_n\}$ of elements of $J(M)$
such that
$\{p_\Sigma(\alpha_n)\}$ is a sequence of distinct elements of $G(\Sigma,Fr(\Sigma))$
and $\alpha_n(K)\cap C$ is non-empty for all $n$.  If a component $\Sigma_i$ of $\Sigma(M)$
is a tiny interval bundle or a solid torus, then $G(\Sigma_i,Fr(\Sigma_i))$ is finite,
by Lemma \ref{Jstructure}.
So, we may pass to
a subsequence, so that there exists an interval bundle $\Sigma_i$ which is not tiny
such that $\{p_i(\alpha_n)\}$ is a sequence
of distinct elements of $G(\Sigma_i,Fr(\Sigma_i))$. Theorem \ref{PSfacts}
then implies that $\{p_i(\alpha_n)(r_i(K))\}$ leaves every compact subset of $X(\Sigma_i)$.
Therefore, since $r_i(\alpha_n(K))=p_i(\alpha_n)(r_i(K))$ for all $n$,
$\{\alpha_n(K)\}$ leaves every compact subset of $X(M)$. This contradiction
establishes claim (2).
\end{proof}

\subsection{Assembly}

Let $W(M)=V(M)\cap Z(M)$. Since $V(M)$ and $Z(M)$ are open ${\rm Out}(\pi_1(M))$-invariant
neighborhoods of $AH_n(M)$, so is $W(M)$.  It remains to prove that ${\rm Out}(\pi_1(M))$
acts properly discontinuously on $W(M)$. Since $J(M)$ is a finite index
subgroup of ${\rm Out}(\pi_1(M))$, it suffices to prove that $J(M)$ acts properly
discontinuously on $W(M)$. We will actually establish the following
stronger fact, which will complete the proof of Theorem \ref{Pdiscnbhd}.

\begin{lemma}
If $K$ is a compact subset of $W(M)$ and $\{\alpha_n\}$ is
a sequence of distinct elements of $J(M)$, then $\{\alpha_n(K)\}$ leaves
every compact subset of $X(M)$.
\end{lemma}

\begin{proof} If our claim fails, then there exists a compact subset
$K$ of $W(M)$, a compact subset $C$ of $X(M)$ and a sequence $\{\alpha_n\}$
of distinct elements of $J(M)$ such that $\alpha_n(K)\cap C$ is non-empty.
We may  pass to an infinite subsequence, still called $\{\alpha_n\}$,
such that either $\{p_\Sigma(\alpha_n)\}$ is a sequence of distinct elements or 
$\{\rho_\Sigma(\alpha_n)\}$ is constant.

If  $\{p_\Sigma(\alpha_n)\}$ is a sequence of distinct elements, Lemma \ref{controlMod} 
immediately implies
that $\{\alpha_n(K)\}$ leaves every compact subset of $X(M)$ and we obtain
the desired contradiction. 

If $\{\rho_\Sigma(\alpha_n)\}$ is constant, then, by Theorem \ref{Jstructure},
there exists a sequence $\{\beta_n\}$ of distinct elements of $K(M)$
such that $\alpha_n=\alpha_1\circ \beta_n$ for all $n$. Lemma \ref{controlK}
implies that $\{\beta_n(K)\}$ exits every compact subset of $X(M)$. Since
$\alpha_1$ induces a homeomorphism of $X(M)$, it follows that
$\{\alpha_n(K)=\alpha_1(\beta_n(K))\}$ also leaves every compact subset of
$X(M)$. This contradiction completes the proof.
\end{proof}

 \end{document}